\documentclass{article}
\usepackage[letterpaper,top=1in,bottom=1in,left=1in,right=1in]{geometry}

\usepackage{amsfonts}       

\usepackage{amsthm}
\usepackage{mathtools}      
\usepackage{amssymb}        
\usepackage{filecontents}
\usepackage{graphicx}       
\usepackage{marginnote}     
\usepackage{marvosym}       
\usepackage{overpic}        
\usepackage{tabularx}
\usepackage{cite}
\usepackage{color}
\usepackage[linesnumbered,algoruled,boxed,lined]{algorithm2e}
\usepackage[normalem]{ulem}
\usepackage{epstopdf}

\usepackage{enumitem}
\usepackage[normalem]{ulem}
\usepackage{lipsum}
\usepackage{comment}
\usepackage{url}

\usepackage{diagbox}
\usepackage{multirow}
\usepackage[dvipsnames]{xcolor}
\newif\ifdraft
\draftfalse

\newif\ifoc 
\octrue

\ifdraft
\usepackage[paperheight=11in,paperwidth=9.5in,
			left=1.25in,right=1.25in,
			top=0.75in,bottom=0.75in,
			heightrounded,marginparwidth=1.2in,
			marginparsep=0.05in]{geometry}
\usepackage{xcolor}
\usepackage{xargs} 
\usepackage[textsize=footnotesize]{todonotes}
\newcommandx{\nt}[2][1=]{\todo[linecolor=red,
			backgroundcolor=red!10,bordercolor=red,#1]{#2}}
\newcommandx{\jy}[2][1=]{\todo[linecolor=green,
			backgroundcolor=green!10,bordercolor=green,#1]{JY: #2}}
\newcommandx{\sw}[2][1=]{\todo[linecolor=blue,
			backgroundcolor=blue!10,bordercolor=blue,#1]{SW: #2}}
\else
\newcommand{\nt}[1]{{}}
\newcommand{\jg}[1]{{}}
\newcommand{\jy}[1]{{}}
\newcommand{\sw}[1]{{}}
\fi

\newif\iftwocolumn
\twocolumntrue

\newtheorem{problem}{Problem}[section]

\newtheorem{corollary}{Corollary}[section]
\newtheorem{theorem}{Theorem}[section]
\theoremstyle{definition}

\theoremstyle{remark}
\newtheorem*{remark}{Remark}

\SetKwProg{Fn}{Function}{}{}
\SetKwComment{Comment}{$\triangleright$\ }{}

\makeatletter
\def\subsubsection{\@startsection{subsubsection}
                                 {3}
                                 {\z@ \hspace*{1mm}}
                                 {0ex plus 0.1ex minus 0.1ex}
                                 {0ex}
                                 {\normalfont\normalsize\itshape}}
\makeatother

\def\R{\mathcal R}
\def\C{\mathcal C}

\def\P{\mathcal P}

\def\W{\mathcal W}
\def\opg{{\sc {OPG}}\xspace}
\def\opglr{{\sc {OPG${}_{LR}$}}\xspace}
\def\opglrd{{\sc {D-OPG${}_{LR}$}}\xspace}
\def\opgmc{{\sc {OPG${}_{MC}$}}\xspace}

\def\twopart{\textbf{\textsc{Partition}}\xspace}
\def\tpart{\textbf{\textsc{$3$-Partition}}\xspace}
\def\ttkp{\textbf{\textsc{Knapsack}}\xspace}
\def\ttukp{\textbf{\textsc{Unbounded Knapsack}}\xspace}

\def\subsetsum{\textbf{\textsc{Subset Sum}}\xspace}

\title{
Optimal Perimeter Guarding with Heterogeneous Robot Teams: 
Complexity Analysis and Effective Algorithms
}
\author{
Si Wei Feng  \qquad Jingjin Yu
\thanks{
S.-W. Feng and J. Yu are with the Department of Computer Science, 
Rutgers, the State University of New Jersey, Piscataway, NJ, 
USA. E-Mails: \{{\tt siwei.feng, jingjin.yu}\}\hspace*{.25em}
\MVAt \hspace*{.25em}rutgers.edu. 
}
}

\begin{document}

\maketitle
\thispagestyle{empty}
\pagestyle{empty}



\begin{abstract}
We perform structural and algorithmic studies of significantly generalized
versions of the optimal perimeter guarding (\opg) problem 
\cite{FenHanGaoYu19RSS}. As compared with the original \opg  where robots 
are uniform, in this paper, many mobile robots with heterogeneous sensing 
capabilities are to be deployed to optimally guard a set of one-dimensional 
segments. Two complimentary formulations are investigated where one limits 
the number of available robots (\opglr) and the other seeks to minimize 
the total deployment cost (\opgmc). In contrast to the original \opg which 
admits low-polynomial time solutions, both \opglr and \opgmc are 
computationally intractable with \opglr being strongly NP-hard. Nevertheless, 
we develop fairly scalable pseudo-polynomial time algorithms for practical, 
fixed-parameter subcase of \opglr; we also develop pseudo-polynomial time 
algorithm for general \opgmc and polynomial time algorithm for the fixed-parameter
\opgmc case. The applicability and effectiveness of selected 
algorithms are demonstrated through extensive numerical experiments.
\end{abstract}

\section{Introduction}\label{sec:intro}
Consider the scenario where many mobile guards (or sensors) are to be deployed 
to patrol 
the perimeter of some 2D regions (Fig.~\ref{fig:ex}) against intrusion, where 
each guard may effectively cover a continuous segment of a region's boundary. 
When part of a boundary need not be secured, e.g., there may already be 
some existing barriers (the blue segments in Fig.~\ref{fig:ex}), optimally 
distributing the robots so that each robot's coverage is minimized becomes 
an interesting and non-trivial computational task \cite{FenHanGaoYu19RSS}. 
It is established \cite{FenHanGaoYu19RSS} that, when the guards have 
the same capabilities, the problem, called the {\em optimal perimeter guarding} 
(\opg), resides in the complexity class P (polynomial time class), 
even when the robots must be distributed across many different boundaries. 

\begin{figure}[ht]
\begin{center}
\begin{overpic}[width={\ifoc 4in \else 3in \fi},tics=5]{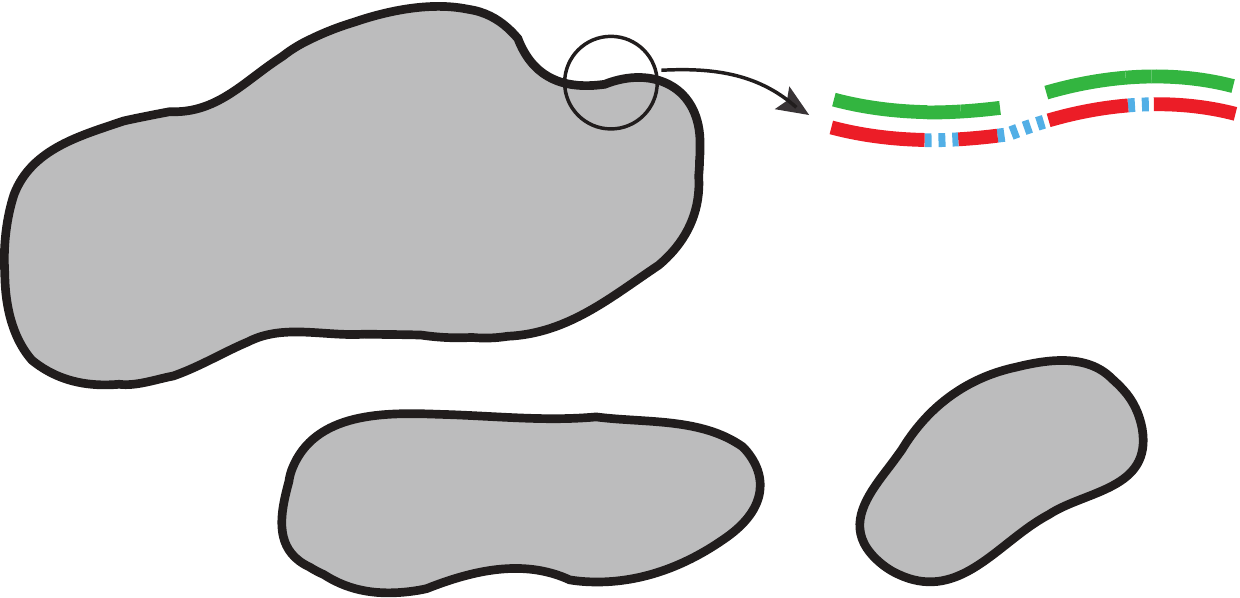}
\end{overpic}
\end{center}
\caption{\label{fig:ex} A scenario where boundaries of three (gray) 
regions must be secured. Zooming in on part of the boundary of one 
of the regions (the part inside the small circle), portions of the 
boundary (the red segments) must be guarded while the rest (the 
blue dotted segments) does not need guarding. For example, the zoomed-in 
part of the boundary may be monitored by two mobile robots, each patrolling
along one of the green segments.}
\end{figure}

In this work, we investigate a significantly more general version of \opg 
where the mobile guards may be heterogeneous. More specifically, two 
formulations with different guarding/sensing models are addressed in our 
study. 
In the first, the number of available robots is fixed where robots of 
different types have a fixed ratio of capability (e.g., one type of 
robot may be able to run faster or may have better sensor). The guarding task 
must be evenly divided among the robots so that each robot, regardless of 
type, will not need to bear a too large coverage/capability ratio. This 
formulation is denoted as {\em optimal perimeter guarding with limited 
resources} or \opglr.
In the second, the number of robots is unlimited; instead, for each type, 
the sensing range is fixed with a fixed associated cost. The goal here is 
to find a deployment plan so as to fully cover the perimeter while minimizing 
the total cost. We call this the {\em optimal perimeter guarding with 
minimum cost} problem, or \opgmc. 

Unlike the plain vanilla version of the \opg problem, we establish that both 
\opglr and \opgmc are NP-hard when the number of robot types is part of the 
problem input. They are, however, at different hardness levels. \opglr is shown 
to be NP-hard in the strong sense, thus reducing the likelihood of finding a fully
polynomial time approximation scheme (FPTAS).
Nevertheless, for the more practical case where the number of robot types 
is a constant, we show that \opglr can be solved using a pseudo-polynomial 
time algorithm with reasonable scalability. On the other hand, we show that 
\opgmc is weakly NP-hard through the establishment of a pseudo-polynomial 
time algorithm for \opgmc with arbitrary number of robot types. 
We further show that, when the number of robot types is fixed, \opgmc can be 
solved in polynomial time through a fixed-parameter tractable (FPT) approach.
This paragraph also summarizes the main contributions of this work. 

A main motivation behind our study of the \opg formulations is to address 
a key missing element in executing autonomous, scalable, and optimal robot 
deployment tasks. Whereas much research has been devoted to multi-robot 
motion planning \cite{ErdLoz86,arai2002advances} with great success, e.g., 
\cite{blm-rvo,smith2009monotonic,ayanian2010decentralized,turpin2014capt,
alonso2015multi,SolYu15}, existing results in the robotics literature appear 
to generally assume that a target robot distribution is already provided; the 
problem of how to effectively generate optimal deployment patterns is largely 
left unaddressed. It should be noted that control-based solutions to the 
multi-agent deployment problem do exist, e.g.,\cite{ando1999distributed,
jadbabaie2003coordination,cortes2004coverage,ren2005consensus,
schwager2009optimal,yu2012rendezvous,morgan2016swarm}, but the final solutions 
are obtained through many local iterations and generally do not come with 
global optimality guarantees. For example, in \cite{cortes2004coverage}, 
Voronoi-based iterative methods compute locally optimal target formations 
for various useful tasks. In contrast, this work, as well as 
\cite{FenHanGaoYu19RSS}, targets the scalable computation of globally optimal 
solutions. 

As a coverage problem, \opg may be characterized as a 1D version 
of the well-studied Art Gallery problems  \cite{o1987art,shermer1992recent},
which commonly assume a sensing model based on line-of-sight 
visibility\cite{lozano1979algorithm}; the goal is to ensure that every point
in the  interior of a given region is visible to at least one of the deployed 
guards. Depending on the exact formulation, guards may be placed on 
boundaries, corners, or the interior of the region. Not surprisingly, Art
Gallery problems are typically NP-hard \cite{lee1986computational}. Other
than Art Gallery, 2D coverage problems with other sensing models, e.g., 
disc-based, have also been considered \cite{thue1910dichteste,hales2005proof,
drezner1995facility,cortes2004coverage,pavone2009equitable
,pierson2017adapting}, where some formulations prevent the overlapping 
of individual sensing ranges \cite{thue1910dichteste,hales2005proof} while 
others seek to ensure a full coverage which often requires intersection
of sensor ranges. 
In viewing of these studies, this study helps painting a broader landscape 
of sensor coverage research.

In terms of structural resemblance, \opglr and \opgmc share many similarities 
with {\em bin packing}  \cite{johnson1973near} and other related problems. 
In a bin packing problem, objects are to be selected to fit within bins of 
given sizes. Viewing the segments (the red ones in Fig.~\ref{fig:ex}) as 
bins, \opg seeks to place guards so that the segments are fully contained in 
the union of the guards' joint coverage span. In this regard, \opg is a dual
problem to bin packing since the former must overfill the bins and the later 
cannot fully fill the bins. In the extreme, however, both bin packing and 
\opg converge to a \subsetsum \cite{karp1972reducibility} like problem where 
one seeks to partition objects into halves of equal total sizes, i.e., the 
objects should fit exactly within the bins. With an additional cost term, 
\opgmc has further similarities with the \ttkp problem \cite{ukphardness}, 
which is weakly NP-hard \cite{dantzig1957discrete}.

The rest of the paper is organized as follows. In Section~\ref{sec:problem},
mathematical formulations of the two \opg variants are fully specified. In
Section~\ref{sec:hardness}, both \opglr and \opgmc are shown to be 
NP-hard. Despite the hardness hurdles, in Section~\ref{sec:algorithm}, 
multiple algorithms are derived for \opglr and \opgmc, including effective
implementable solutions for both. In Section~\ref{sec:application},
we perform numerical evaluation of selected algorithms and demonstrate 
how they may be applied to address multi-robot deployment problems. We 
discuss and conclude our study in Section~\ref{sec:conclusion}. Please
see \url{https://youtu.be/6gYL0_B3YTk} for an illustration of the problems 
and selected instances/solutions. 

\section{Preliminaries}\label{sec:problem}
Let $\W \subset \mathbb R^2$ be a compact (closed and bounded) 
two-dimensional workspace. There are  $m$ pairwise disjoint {\em 
	regions} $\R = \{R_1, \ldots, R_m\}$ where each region $R_i \subset \W$ 
is homeomorphic to the closed unit disc, i.e., there exists a continuous 
bijection $f_i: R_i \to \{(x, y) \mid x^2 + y^2 \le 1\}$ for all $1 \le 
i \le m$. For a given region $R_i$, let $\partial R_i$ be its (closed) 
boundary (therefore, $f_i$ maps $\partial R_i$ to the unit circle  
$\mathbb S^1$). With a slight abuse of notation, define $\partial \R 
= \{\partial R_1, \ldots, \partial R_m\}$. Let $P_i \subset \partial R_i$ 
be the part of $\partial R_i$ that is accessible, specifially, not blocked by 
obstacles in $\W$. This means that each $P_i$ is either a single closed 
curve or formed by a finite number of (possibly curved) line segments. 
Define  $\P = \{P_1, \ldots, P_m\} \subset \W$ as the {\em perimeter} 
of $\R$ which must be {\em guarded}. More formally, each $P_i$ is 
homeomorphic to a compact subset of the unit circle (i.e., it is 
assumed that the maximal connected components of $P_i$ are closed 
line segments). For a given $P_i$, each one of its maximal connected 
component is called a {\em perimeter segment} or simply a {\em segment}, 
whereas each maximal connected component of $\partial R_i \backslash P_i$ 
is called a {\em perimeter gap} or simply a {\em gap}. An example setting is 
illustrated in Fig.~\ref{fig:example-boundaries} with two regions. 

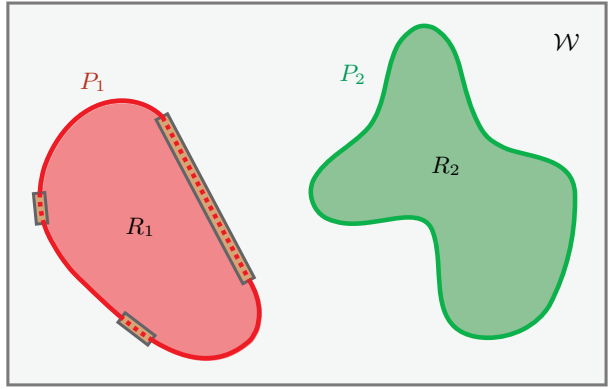
\begin{figure}[ht]
	\begin{center}
		\begin{overpic}[width={\ifoc 4in \else 3in \fi},tics=5]
			{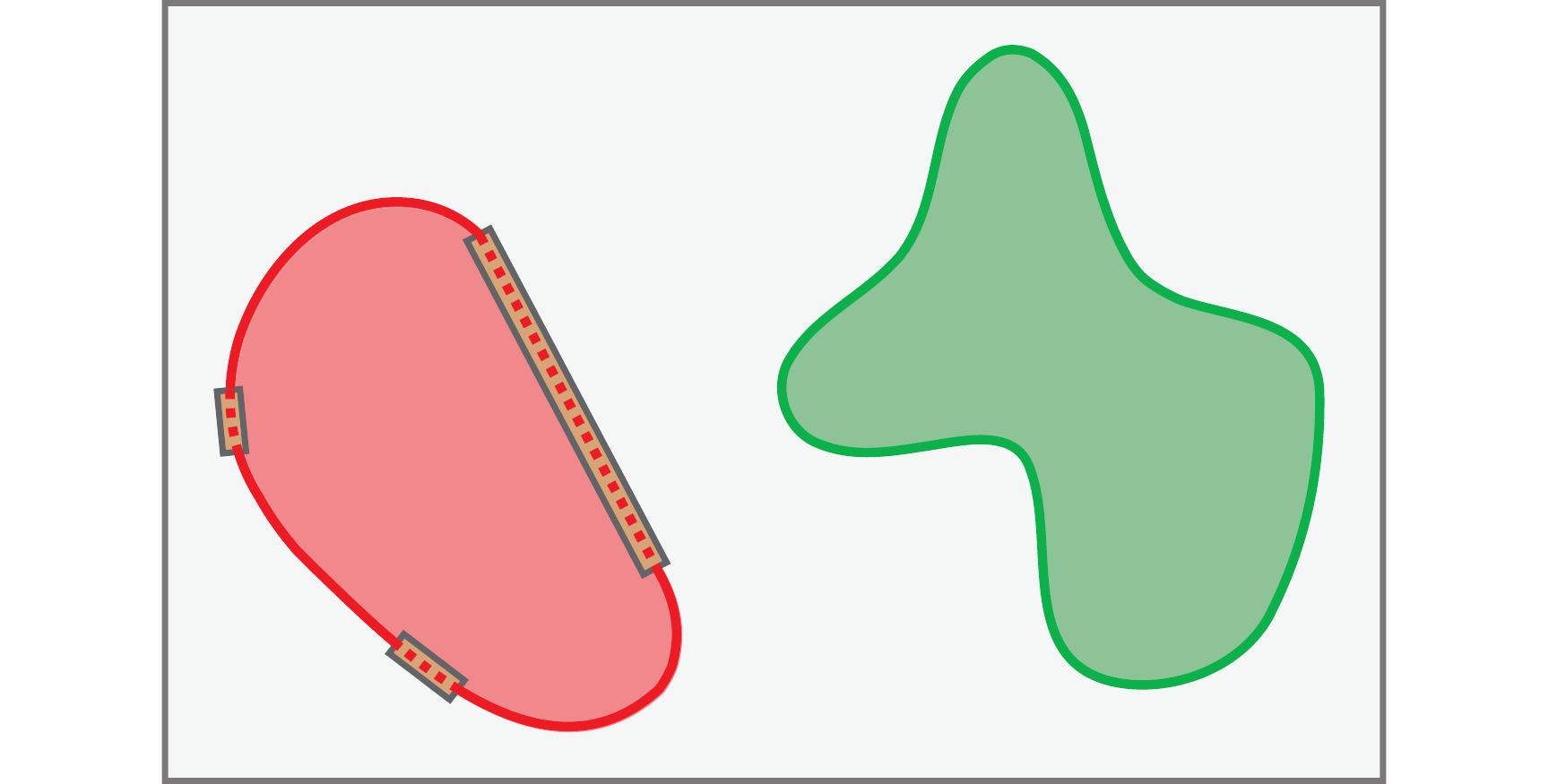}
			\put(26,20){{\small $R_1$}}
			\put(20,39){{\small \textcolor{BrickRed}{$P_1$}}}
			\put(66,28){{\small $R_2$}}
			\put(54,40){{\small \textcolor{ForestGreen}{$P_2$}}}
			\put(82,44){{\small $\W$}}
		\end{overpic}
	\end{center}
	\caption{\label{fig:example-boundaries} An example of a workspace $\W$ 
		with two regions $\{R_1, R_2\}$. Due to three {\em gaps} on $\partial R_1$, 
		marked as dotted lines within long rectangles, $P_1 \subset \partial R_1$ 
		has three {\em segments} (or maximal connected components); $P_2 = \partial 
		R_2$ has a single segment with no gap.}
\end{figure}

After deployment, some number of robots are to {\em cover} the perimeter 
$\P$ such that a robot $j$ is assigned a continuous closed subset $C_j$ 
of some $\partial R_i, 1 \le i \le m$. All of $\P$ must be {\em covered} 
by $\C$, i.e., 
$\bigcup_{P_i \in \P} P_i  \subset \bigcup_{C_j \in \C} C_j$,
which implies that elements of $\C$ need not intersect on their interiors. 
Hence, it is assumed that any two elements of $\C$ may share at most their 
endpoints. Such a $\C$ is called a {\em cover} of $\P$. Given a cover 
$\C$, for a $C_j \in \C$, let $len(C_j)$ denote its length (more formally, 
measure). 

To model heterogeneity of the robots, two models are explored in this
study. In either model, there are $t$ types of robots. In the first model,
the number of robots of each type is fixed to be $n_1, \ldots, n_t$ with 
$n = n_1 + \cdots + n_t$. For a robot $1 \le j \le n$, let $\tau_j$ denote 
its type. Each $1 \le \tau \le t$ type of robots has some 
level of {\em capability} or {\em ability} $a_{\tau} \in \mathbb Z^+$. We 
wish to balance the load among all robots based on their capabilities, 
i.e., the goal is to find cover $\C$ for all robots such that the quantity 
\[
\max_{C_j \in \C} \frac{len(C_j)}{a_{\tau_j}},
\]
which represents the largest coverage-capacity ratio, is minimized. 
We note that when all capacities are the same, e.g., $a_{\tau} = 1$ for 
all robots, this becomes the standard \opg problem studied in \cite{FenHanGaoYu19RSS}. 
We call this version of the perimeter guarding problem {\em optimal 
	perimeter guarding with limited resources} or \opglr. The formal 
definition is as follows.

\begin{problem}[Optimal Perimeter Guarding with Limited Resources 
	(\opglr)] Let there be $t$ types of robots. For each type $1\le \tau 
	\le t$, there are $n_{\tau}$ such robots, each having the same 
	capability parameter $a_{\tau}$. Let $n = n_1 + \cdots + n_t$. 
	Given the perimeter set $\P = \{P_1, \ldots, P_m\}$ of a set of 
	2D regions $\R =\{R_1, \ldots, R_m\}$, find a set of $n$ continuous 
	line segments $\C^* = \{C_1^*, \ldots, C_n^*\}$ such that $\C^*$ covers 
	$\P$, i.e., \begin{align}\label{eq:coverage}
	\bigcup_{P_i \in \P} P_i  \subset \bigcup_{C_j^* \in \C^*} C_j^*,
	\end{align}
	such that a $C_j^*$ is covered by robot $j$ of type $\tau_j$, and such that,
	among all covers $\C$ satisfying~\eqref{eq:coverage}, 
	\begin{align}\label{eq:objective}
	\C^* = \underset{\C}{\mathrm{argmin}} \max_{C_j \in \C} 
	\frac{len(C_j)}{a_{\tau_j}}.
	\end{align}
\end{problem}

Whereas the first model caps the number of robots, the second
model fixes the maximum coverage of each type of robot. That is, for 
each robot type $1 \le \tau \le t$, $n_{\tau}$, the number of robots of type $\tau$,
is unlimited as long as it is non-negative, but each such robot can only cover 
a maximum length of $\ell_{\tau}$. 
At the same time, using each such robot incurs a cost of $c_{\tau}$. The 
goal here is to guard the perimeters with the minimum total cost. We 
denote this problem {\em optimal perimeter guarding with minimum 
	cost} or \opgmc. 

\begin{problem}[Optimal Perimeter Guarding with Minimum Cost
	(\opgmc)] Let there be $t$ types of robots of unlimited quantities. 
	For each robot of type $1\le \tau \le t$, it can guard a length of 
	$\ell_{\tau}\in\mathbb{Z^+}$ with a cost of $c_{\tau}\in\mathbb{Z^+}$. Given the perimeter set
	$\P = \{P_1, \ldots, P_m\}$ of a set of 2D regions $\R =\{R_1, \ldots, 
	R_m\}$, find a set of $n = n_1 + \cdots + n_t$ continuous line segments 
	$\C^* = \{C_1^*, \ldots, C_n^*\}$ where $n_{\tau}$ such segments are 
	guarded by type $\tau$ robots, such that $\C^*$ covers $\P$, i.e., 
	\begin{align}\label{eq:coverage2}
	\bigcup_{P_i \in \P} P_i  \subset \bigcup_{C_j^* \in \C^*} C_j^*,
	\end{align}
	such that a $C_j^*$ is covered by robot $j$ of type $\tau_j$, i.e., 
	$C_j^* \le \ell_{\tau_j}$, and such that,
	among all covers $\C$ satisfying~\eqref{eq:coverage2}, 
	\begin{align}\label{eq:objective2}
	\C^* = \underset{\C}{\mathrm{argmin}} \sum_{1 \le \tau \le t} 
	n_{\tau}c_{\tau}.
	\end{align}
\end{problem}

\section{Computational Complexity for Variable Number of Robot Types}\label{sec:hardness}
We explore in this section the computational complexity of \opglr 
and \opgmc. Both problems are shown to be NP-hard with \opglr 
being strongly NP-hard. We later confirm that \opgmc is weakly 
NP-hard (in Section~\ref{sec:algorithm}).

\subsection{Strong NP-hardness of \opglr}\label{subsec:opglr-hardness}
When the number of types $t$ is a variable, i.e., $t$ is not a constant
and may be arbitrarily large,
\opglr is shown to be NP-hard via the reduction from \tpart \cite{garey1975complexity}:

\vspace*{1mm}
\noindent
PROBLEM: \tpart\\
INSTANCE: A finite set A of $3m$ elements, a bound $B\in \mathbb{Z^+}$, 
and a ``size'' $s(a)\in \mathbb{Z^+}$ for each $a\in A$,
such that each $s(a)$ satisfies $B/4 < s(a) <B/2$ and $\sum_{a\in A} s(a) = mB$.\\
QUESTION: Is there a partition of $S$ into $m$ disjoint subsets $S_1, 
\ldots, S_m$ such that for $1\leq i\leq m$, 
$\sum_{a\in S_i} s(a) = B$?
\vspace*{1mm}

\tpart is shown to be NP-complete in the strong sense\cite{GarJoh79}, 
i.e., it is NP-complete even when all numeric inputs are bounded by a polynomial 
of the input size. 

For the reduction, it is more convenient to work with a decision 
version of the \opglr problem, denoted as \opglrd. In the \opglrd 
problem, $a_{\tau}$ is the actual length robot type $\tau$ covers. 
That is, the coverage length of a robot is fixed. The \opglrd problem 
is specified as follows. 

\vspace*{1mm}
\noindent
PROBLEM: \opglrd\\
INSTANCE: $t$ types of robots where there are $n_{\tau}$ robots for 
each type $1 \le \tau \le t$; $n = n_1 + \cdots + n_t$. A robot of 
type $\tau$ has a coverage capacity $a_{\tau}$. A set of perimeters 
$\P = \{P_1, \ldots, P_m\}$ of a set of 2D regions 
$\R =\{R_1, \ldots, R_m\}$.\\ 
QUESTION: Is there a deployment of $n$ disjoint subsets $C_1, \ldots, C_n$
of $\{\partial R_1, \ldots, \partial R_m\}$ such that 
$P_1 \cup \ldots \cup P_m \subset C_1 \cup \ldots \cup C_n$, where
$C_j$ is a continuous segment for all $1 \le j \le n$, and for each 
$1 \le j \le n$, there is a unique robot whose type $\tau$, $1 \le \tau 
\le t$ satisfies $a_{\tau} \ge len(C_j)$?
\vspace*{1mm}

\begin{theorem}\label{t:opglr-hard}
	\opglr is strongly NP-hard. 
\end{theorem}
\begin{proof}
	A polynomial reduction from \tpart to \opglrd is constructed
	by a restriction of \opglrd. Given a \tpart instance with former notations,
	we apply several restrictions on \opglrd: {\em (i)} there are $3m$ types of robot
	and there is a single robot 
	for each type, i.e., $n_{\tau} = 1$ for $1 \le \tau \le t$, so $n=t=3m$ 
	{\em (ii)} the $3m$ capacities $a_1, \ldots, a_{3m}$ are set to be equal to
	$s(a)$ for each of the $3m$ elements $a\in A$, and {\em (iii)} 
	there are $3m$ perimeters and each perimeter $P_i$ is continuous and
	$len(P_i)=B$ for all $1 \le i \le m$.

	With the setup, the reduction proof is straightforward. Clearly, the 
	\tpart instance admits a partition of $A$ into $S_1, \ldots, 
	S_m$ such that $\sum_{a \in S_i} s(a) = B$ for all $1\leq i\leq m$ 
	if and only if a
	valid depolyment exists in the corresponding \opglrd instance. 
	It is clear that the reduction from \tpart to 
	\opglrd is polynomial (in fact, linear). Based on the 
	reduction and because \tpart is strongly NP-hard, so is \opglrd 
	and \opglr.
\end{proof}

\begin{remark}
	One may also reduce weakly NP-hard problems, e.g., \twopart
	\cite{karp1972reducibility}, to \opglr for variable number of robot 
	types $t$. Being strongly NP-hard, \opglr is unlikely to admit pseudo-polynomial 
	time solutions for variable $t$. This contrasts with a later result 
	which provides a pseudo-polynomial time algorithm for \opglr for 
	constant $t$, as one might expect in practice where robots have limited 
	number of types. We also note that Theorem~\ref{t:opglr-hard} continues 
	to hold for a single perimeter with multiple segments, each 
	having a length $B$ in previous notation, separated by ``long'' gaps. Obviously, 
	\opglrd is in NP, thus rendering it NP-complete. 
\end{remark}

\subsection{NP-hardness of \opgmc}
The minimum cost \opg variant, \opgmc, is also NP-hard, which may be 
established through reduction from the \subsetsum problem 
\cite{karp1972reducibility}:

\vspace*{1mm}
\noindent
PROBLEM: \subsetsum \\
INSTANCE: A set $B$ with $|B| = n$ and a weight function $w: B \to 
\mathbb Z^+$, and an integer $W$.\\ 
QUESTION: Is there a subset $B' \subseteq B$ such that $\sum_{b \in B'} 
w(b) = W$?
\vspace*{1mm}

\begin{theorem}\label{t:opgmc-hard}
	\opgmc is NP-hard. 
\end{theorem}
\begin{proof}
	Given a \subsetsum instance, we construct an \opgmc instance with a 
	single perimeter containing a single segment with length $L$ to be 
	specified shortly. Let there be $t=2n$ types of robots. For $1 \le i 
	\le n$, let robot type $2i-1$ have $\ell_{2i-1} = c_{2i-1} = 
	w(b_i) + (2^{n + 1} + 2^i)W'$ and let robot type $2i$ have $\ell_{2i} 
	= c_{2i} = (2^{n + 1} + 2^i)W'$. Here, $W'$ can be any integer number no less than 
	$\sum_{b\in B} w(b)$. Set $L = W + (n2^{n+1} + 2^n + \ldots 
	+ 2^1)W'$. We ask the ``yes'' or ``no' decision question of whether there 
	are robots that can be allocated to have a total cost no more than $L$ 
	(equivalently, equal to $L$, as the cost density $c_\tau/l_\tau$
	is always $1$).
	
	Suppose the \subsetsum instance has a yes answer that uses a subset
	$B' \subseteq B$. Then, the \opgmc instance has a solution with cost $L$ 
	that can be constructed as follows. For each $1 \le i \le n$, a single 
	robot of type $2i - 1$ is taken if $b_i \in B'$. Otherwise, a single 
	robot of type $2i$ is taken. This allocation of robots yields a total 
	length and cost of $L$. 
	
	For the other direction, we first show that if the \opgmc instance 
	is to be satisfied, it can only use a single robot from type $2i-1$ 
	or $2i$ for all $1 \le i \le n$. First, if more than $n$ robots are 
	used, then the total cost exceeds $(n + 1)2^{n+1}W' > L$ as $W\leq W'$.
	Similarly, if less than $n$ 
	robots are used, the total length is at most 
	$(n-1)2^{n+1}W'+(2^{n+1}-1)W'+ W' < L$. 
	Also, to match the $(2^n + \ldots + 2)W'$ part of the cost, exactly one robot 
	from type $2i-1$ or $2i$ for all $1 \le i \le n$ must be taken. 
	Now, if the \opgmc decision instance has a yes answer, if a robot 
	of type $2i -1$ is used, let $b_i \in B$ be part of $B'$, which 
	constructs a $B'$ that gives a yes answer to the \subsetsum instance.
\end{proof}
\begin{remark}
	It is also clear that the decision version of the \opgmc problem is 
	NP-complete. The \subsetsum is a weakly NP-hard problem that admits 
	a pseudo-polynomial time algorithm \cite{dantzig1957discrete}. 
	As it turns out, \opgmc, which shares similarities with \subsetsum
	and \ttkp (in particular, \ttukp \cite{ukphardness}), though NP-hard, 
	does admit a pseudo-polynomial time algorithm as well. 
\end{remark}

\section{Exact Algorithms for \opglr and \opgmc}\label{sec:algorithm}
In this section, we describe three exact algorithms for solving the two 
variations of the \opg problem. First, we present a pseudo-polynomial time
algorithm for \opglr when the number of robot types, $t$, is a fixed constant. 
Given that \opglr is strongly NP-hard, this is in a sense a best possible 
solution. 
For \opgmc, in addition to providing a pseudo-polynomial algorithm for 
arbitrary $t$, which confirms that \opgmc is weakly NP-hard, we also provide
a polynomial time approximation scheme (PTAS). We then further show the 
possibility of solving \opgmc in polynomial time when $t$ is a fixed constant. 
We mention that our development in this section focuses on the single 
perimeter case, i.e., $m = 1$, as the generalization to arbitrary $m$ is 
straightforward using techniques described in \cite{FenHanGaoYu19RSS}. With this in mind, 
we also provide the running times for the general setting with arbitrary $m$ 
but refer the readers to \cite{FenHanGaoYu19RSS} on how these running times can be derived. 

For presenting the analysis and results, for the a perimeter $P$ that we work 
with, assume that it has $q$ perimeter segments $S_1, \ldots, S_q$ that need 
to be guarded; these segments are separated by $q$ gaps $G_1, \ldots, G_q$. 
For $1 \le i, i' \le q$, define $S_{i\sim i'} = S_i \cup G_i \cup S_{i+1} \cup \ldots 
\cup G_{i'-1} \cup S_{i'}$ where $i'$ may be smaller than $i$ (i.e., $S_{i\sim i'}$
may wrap around $G_q$),
For the general case with $m$ perimeters, assume that a perimeter $P_i$ has
$q_i$ segments. 

\subsection{Pseudo-Polynomial Time Algorithm for \opglr with Fixed Number
of Robot Types}
\SetKw{Continue}{continue}
\SetKw{True}{true}
\SetKw{False}{false}
\SetKwComment{Comment}{\%}{}
\SetKwInOut{Input}{Input}
\SetKwInOut{Output}{Output}
\def\inc{{\sc Inc}\xspace}
\def\knapsack{\textbf{\textsc{Knapsack}}\xspace}
\def\opglrfeasible{{\sc OPG-lr-Feasible}\xspace}
\def\opgmcdp{{\sc OPG-mc-DP}\xspace}
We set to develop an algorithm for \opglr for arbitrary $t$, the number of robot 
types; the algorithm runs in pseudo-polynomial time when $t$ is a constant. 
At a higher level, our proposed algorithm works as follows. First, our main effort 
here goes into deriving a feasibility test for \opglrd as defined in 
Section~\ref{subsec:opglr-hardness}. With such a feasibility test, we can then 
find the optimal $\frac{len(C_j)}{a_{\tau_j}}$ in \eqref{eq:objective} via binary search.
Let us denote the optimal value of $\frac{len(C_j)}{a_{\tau_j}}$ as $\ell^*$. 

\subsubsection{Feasibility Test for \opglrd} The feasibility test for \opglrd 
essentially tries different candidate $\ell$ to find $\ell^*$. Our implementation uses 
ideas similar to the pseudo-polynomial time algorithm for the \knapsack problem which
is based on dynamic programming (DP). In the test, we work with a fixed starting point on 
$P$, which is set to be the counterclockwise end point of a segment $S_i$, $1 \le i \le q$. 
Essentially, we maintain a $t$ dimensional array $M$ where dimension $\tau$
has a size of $n_{\tau} +1$. An element of the array, $M[n_1']\ldots[n_t']$, holds the 
maximal distance starting from $S_i$ that can be covered by $n_1'$ type 1 robots, 
$n_2'$ type 2 robots, and so on. The DP procedure \opglrfeasible($i, \ell$), outlined in 
Algorithm~\ref{algo:opglrd}, incrementally builds this array $M$. For convenience,
in the pseudo code, $M[\vec{x}]$ denotes an element of $M$ with $\vec{x}$ being 
a $t$ dimensional integer vector. 

\begin{algorithm}\label{algo:opglrd}
	\DontPrintSemicolon
	\KwData{$n_1, \ldots, n_t$, $a_1, \ldots, a_t$,
		$S_1, \ldots, S_q$, $G_1, \ldots, G_q$}
	\KwResult{\True or \False,  indicating whether $S_1, \ldots, S_q$ can be covered}
		Initialize $M$ as a $t$ dimensional array with dimension $\tau$ having a size of $n_{\tau} + 1$;\;
		$\ell_{\tau} \leftarrow a_{\tau}\ell$ for all $1\le \tau \le t$;\;
		\For{$ \vec{x} \in [0, n_1]\times\dots\times[0,n_t]$}{
            $M[\vec{x}]\leftarrow 0$;\;
			\For{$j = 1$ \KwTo $t$}{
				\lIf{$\vec{x}_j = 0$}{\Continue;}
				$\vec{x'}\leftarrow\vec{x}$; $\vec{x'}_j \leftarrow \vec{x'}_j - 1$;\;
				$M[\vec{x}]\leftarrow max$($M[\vec{x}]$, \inc($M[\vec{x'}], \ell_j$));\;
			}
		}
		\Return{$M[n_1]\ldots[n_t] \ge len(S_{i\sim {i-1}})$};
	\caption{\opglrfeasible($i, \ell$)}
\end{algorithm}

In Algorithm~\ref{algo:opglrd}, the procedure \inc($L, \ell$) checks how much of the 
perimeter $P$ can be covered when an additional coverage length $\ell$ is added, 
assuming that a distance of $L$ (starting from some $S_i$) is already covered. An 
illustration of how \inc($L, \ell$) works is given in Fig.~\ref{fig:inc}.

\begin{figure}[ht]
	\begin{center}
		\begin{overpic}[width={\ifoc 4in \else 3in \fi},tics=5]
			{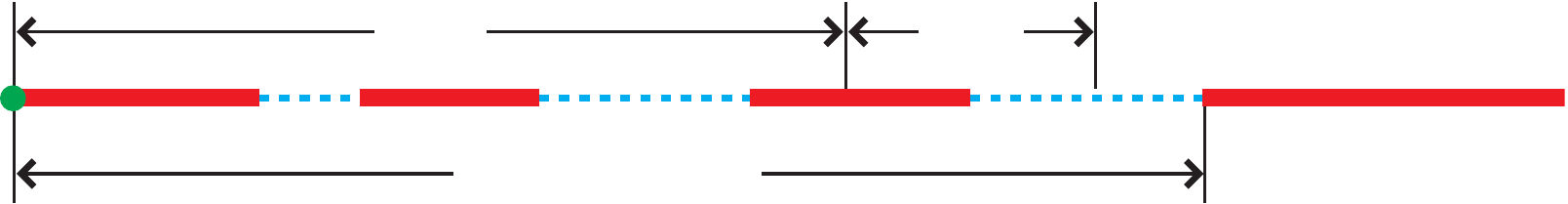}
			\put(26,10){{\small $L$}}
			\put(61,10){{\small $\ell$}}
			\put(30,1){{\small \textsc{Inc}($L$, $\ell$)}}
		\end{overpic}
	\end{center}
    \caption{\label{fig:inc}Suppose starting from the fixed left point, a 
		length of $L$ on the boundary is successfully guarded by a group of 
		robots. Then, a robot	with coverage capacity $\ell$ is appended to 
		the end of the group of robots to increase the total guarded distance. 
		In the figure, the added additional capacity $\ell$ can fully cover 
		the third red segment plus part of the third (dashed) gap. Because 
		there is no need to cover the rest of the third gap, 
		\textsc{Inc}($L$, $\ell$) extends to the end of the gap.}
\end{figure}



By simple counting, the complexity of the algorithm is
$O(q\cdot t\cdot\Pi_{\tau=1}^{t} (n_{\tau}+1))$. However, the amortized
complexity of \inc($\cdot$) for each $\tau$ is $O(q+n_{\tau})$; the algorithm thus 
runs in $O(t\cdot\Pi_{{\tau}=1}^{t}(n_{\tau}+1)+q\cdot\sum_{{\tau}=1}^{t} 
\Pi_{{\tau'}\neq {\tau}} (n_{\tau'} +1))$, 
which is pseudo-polynomial for fixed $t$. After trying every possible starting 
position $i$ with \opglrfeasible($i, \ell$), for a fixed candidate $\ell$, 
\opglrd is solved in $O(q \cdot t\cdot\Pi_{{\tau}=1}^{t}(n_{\tau}+1) + 
q^2\cdot\sum_{{\tau}=1}^{t} \Pi_{{\tau'}\neq {\tau}} (n_{\tau'} +1))$.

\subsubsection{Solving \opglr using Feasibility Test for \opglrd}
Using \opglrfeasible($i, \ell$) as a subroutine to check feasibility for a given $\ell$, 
bisection can be applied over candidate $\ell$ to obtain $\ell^*$. For completing 
the algorithm, one needs to establish when the bisection will stop (notice that, 
even though we assume that $a_\tau \in \mathbb{Z^+}$, for each $1\leq \tau \leq t$, $\ell^*$ need not be 
an integer). 

To derive the stop criterion, we note that given the optimal $\ell^*$, there must 
exist some $S_{i\sim i'}$ that is ``exactly'' spanned by the allocated robots.
That is, assume that $S_{i\sim i'}$ is covered by $n_1'$ of type $1$ robots
and $n_2'$ of type $2$ robots, and so on, then 
\begin{align}\label{eq:exact}
\ell^* = \frac{len(S_{i\sim i'})}{\sum_{1 \le \tau \le t} a_{\tau}\cdot n_{\tau}'}.
\end{align}

\eqref{eq:exact} must hold for some $S_{i\sim i'}$ because if not, the solution 
is not tight and can be further improved. Therefore, the bisection process for 
locating $\ell^*$ does not need to go on further after reaching a certain granularity\cite{FenHanGaoYu19RSS}.
With this established, using 
similar techniques from \cite{FenHanGaoYu19RSS} (we omit the technical detail as it is quite 
complex but without additional new ideas beyond beside what is already covered 
in \cite{FenHanGaoYu19RSS}), we could prove that the full algorithm needs 
no more than $O(q\log(\sum_{\tau}n_{\tau}+q)$ calls to \opglrfeasible($i, \ell$).
This directly implies that \opglr also admits a pseudo-polynomial algorithm for fixed $t$.
\subsubsection{Multiple Perimeters}
Also using techniques developed in \cite{FenHanGaoYu19RSS}, the single perimeter 
result can be readily generalized to multiple perimeters. We omit the mechanical
details of the derivation and point out that the computational complexity in this case becomes
$\tilde{O}( (m-1)\cdot((\Pi_{\tau=1}^t n_\tau) / \max_\tau n_\tau)^2 + 
\sum_{k=1}^{m} (t\cdot q_k\cdot \Pi_{{\tau}=1}^{t}(n_{\tau}+1)+
q_k^2\sum_{{\tau}=1}^{t} \Pi_{{\tau'}\neq {\tau}} (n_{\tau'} +1)))$.

\subsection{Polynomial Time Algorithm for \opgmc with Fixed Number of Robot Types}
The solution to \opgmc will be discussed here. A method based on DP 
will be provided first, which leads to a polynomial time algorithm for a 
fixed number of robot types and a pseudo-polynomial time algorithm when the number 
of robot types is not fixed. For the latter case, a polynomial time approximation 
scheme (PTAS) will also be briefly described.
\subsubsection{Dynamic Programming Procedure for \opgmc}
\def\sol{{\sc Sol}}
\def\presol{{\sc PreSolve}}
When no gaps exist, the optimization problem becomes a covering 
problem as follows. Let $c_{\tau}$, $\ell_{\tau}$, $n_{\tau}$ correspond to the cost, 
coverage length, and quantity of robot type ${\tau}$, respectively, and let total 
length to cover be $L$. We are to solve the optimization problem
\begin{align}\label{eq:ip}
    \min \sum_{\tau} c_{\tau} \cdot n_{\tau} \quad s.t.\, \quad
    \sum_{\tau} \ell_{\tau} \cdot n_{\tau} \geq L, n_{\tau}\geq 0.
\end{align}

Let the solution to the above integer programming problem be \sol($L$).
Notice that, for $S_{i\sim i'}:=\{S_i,\ G_i, \dots, 
G_{i'-1}, S_{i'}\}$, the minimum cost cover is by either: {\em (i)} 
covering the total boundary without skipping any gaps, 
or {\em (ii)} skipping or partially covering some gap, for example $G_k, 
i \le k \le j-1$.
In the first case, the minimum cost is exactly \sol$(\lceil len(S_{i\sim(i+k)}\rceil)$.
In the second case, the optimal structure for the two subsets of perimeter 
segments $S_{i\sim k}$ and $S_{(k+1)\sim j}$ still holds. This means that the 
continuous perimeter segments $S_{i\sim j}$ can be divided into two parts, 
each of which can be treated separately. This leads to a DP approach for \opgmc.
With $M[i][j]$ denoting the minimum cost to cover $S_{i\sim j}$, the DP recursion 
is given by
\[
	\scalebox{0.93}{$M[i][j] = \min(\textit{\sol}(\lceil len(S_{i\sim j})\rceil), \displaystyle\min_k(M[i][k]+M[k+1][j]))$}
\]

The DP procedure is outlined in Algorithm~\ref{alg:opgmc}. In the pseudo code, 
it is assumed that indices of $M$ are modulo $q$, e.g., $M[2][q+1] 
\equiv M[2][1]$. $tmp$ is a temporary variable. 
\begin{algorithm}
    \DontPrintSemicolon
    \KwData{$\ell_1, \dots, \ell_t$, $c_1, \ldots, c_t$,
    	$S_1, \ldots, S_q$, $G_1, \ldots, G_q$}
    \KwResult{$c^*$, the minimum covering cost}
    $M \leftarrow$ a $q\times q$ matrix; $c^* \leftarrow \infty$; \;
    \For{$k \leftarrow 0$ \KwTo $q-1$}{
        \For{$i\leftarrow 1 $ \KwTo $q$}{
            $tmp \leftarrow\ $\sol$(\lceil len(S_{i\sim(i+k)}) \rceil)$; \;
            \For{$j\leftarrow i$ \KwTo $i+k-1$}{
                $tmp \leftarrow \min(tmp, M[i][j] + M[j+1][i+k])$;
            }
            $M[i][i+k] \leftarrow c$;\;
            \lIf{$k = q-1$}{$c^* \leftarrow \min(c^*, M[i][i+k])$;}
        }
    }
    \Return{$c^*;$}
    \caption{\opgmcdp}
    \label{alg:opgmc}
\end{algorithm}

\subsubsection{A Polynomial Time Algorithm for \opgmc for a Fixed Number of Robot Types}
We mention briefly that, by a result of Lenstra \cite{len83}, the optimization problem 
~\eqref{eq:ip} is in P (i.e., polynomial time) when $t$ is a constant. The running time of 
the algorithm \cite{len83} is however exponential in $t$. 

\subsubsection{A Pseudo-polynomial Time Algorithm for Arbitrary $t$}
As demonstrated in the hardness proof, similarities exist between \opg and the \knapsack 
problem. The connection actually allows the derivation of a pseudo-polynomial time algorithm
for arbitrary $t$. To achieve this, we use a routine to pre-compute \sol($L$), called 
\presol(), which is itself a DP procedure similar to that for the \knapsack problem. The 
pseudo code of \presol() is given in Algorithm~\ref{algo:presol}.
\presol() runs in time $O(t\cdot\lceil len(\partial R)\rceil))$. Overall, 
Algorithm~\ref{alg:opgmc} then runs in time $O(q^3+t\cdot\lceil len(\partial R\rceil))$.

\begin{algorithm}\label{algo:presol}
	\DontPrintSemicolon
	\KwData{$\ell_1, \ldots, \ell_t$, $c_1, \ldots, c_t$}
	\KwResult{A lookup table for retrieving \sol($L$)}
		$I_{max} = \lceil len(\partial R)\rceil$; \Comment{\small $I_{max}$ is an integer.}
		$M' \leftarrow$ an array of length $I_{max} + 1$; 
		$M'[0]\leftarrow 0$;\;
		\For{$L \leftarrow$ $1$ \KwTo $I_{max}$}{
			$M'[L]\leftarrow \infty$; \;
			\For{${\tau}\leftarrow 1$ \KwTo $t$}{
				$tmp \leftarrow (L<\ell_{\tau}\ ?\ 0\ :\ M'[L-\ell_{\tau}]) + c_{\tau}$;\;
				$M'[L] \leftarrow min(M'[L], tmp)$;\;
			}
		}
		\Return{$M'$}
	\caption{{\sc PreSolve}}
\end{algorithm}

With the establishment of a pseudo-polynomial time algorithm for \opgmc, we  
have the following corollary. 
\begin{corollary}
	\opgmc is weakly NP-hard. 
\end{corollary}

\subsubsection{FPTAS for Arbitrary $t$}
When the number of robot types is not fixed, Lenstra's algorithm\cite{len83} or 
its variants no longer run in polynomial time. We briefly mention that, 
by slight modifications of a FPTAS for \ttukp problem from \cite{ukpfptas}, a FPTAS 
for \opgmc can be obtained that runs in time $O(q^3 + q^2 \cdot \frac{t}{\epsilon^3})$, 
where $(1+\epsilon)$ is the approximation ratio for both \opgmc and ~\eqref{eq:ip}. 

\subsubsection{Multiple Perimeters} For \opgmc, when there are multiple 
perimeters, e.g., $P_1, \ldots, P_m$, a optimal solution can be obtained 
by optimally solving \opgmc for each perimeter $P_i$ individually and 
then put together the solutions.


\section{Performance Evaluation and Applications}\label{sec:application}
In this section, we provide examples illustrating the typical 
optimal solution structures of \opglr and \opgmc computed by 
our DP algorithms. Using an application scenario, solutions to 
\opglr and \opgmc are also compared. Then, computational 
results from extensive numerical evaluations are presented, 
confirming the effectiveness of these algorithms. The 
implementation is done using the Python and all 
computations are performed on an Intel(R) Core(TM) i7-7700 CPU@3.6GHz 
with 16GB RAM. 

\subsection{Basic Optimal Solution Structure}
Fig.~\ref{fig:opglrm} shows the typical outcome of solving an \opglr 
instance with two perimeters ($m = 2$) for two types of robots with 
$n_1 = 3, a_1 = 5$, and $n_2 = 5, a_2 = 8$. 
In the figure, the red segments are parts of the two perimeters that 
must be guarded. The three orange (resp., five green) segments across 
the two perimeters indicate the desired coverage regions of the three 
(resp., five) type $1$ (resp., type $2$) robots. These coverage regions 
correspond to the optimal solution returned by the DP algorithm. As 
may be observed, the optimal solution is somewhat complex with robots 
of both types on each of the two perimeters; a gap on the second boundary 
also gets covered. The coverage lengths for a robot type are generally 
different; this is due to adjustments that shrink some robots' coverage. 
For example, the first perimeter has a very short orange cover because 
the corresponding perimeter segment is short and gaps around it need 
not be covered (The adjustment procedure is also shown in the video). 
\begin{figure}[!ht]
    \centering
    \includegraphics[scale = 0.6]{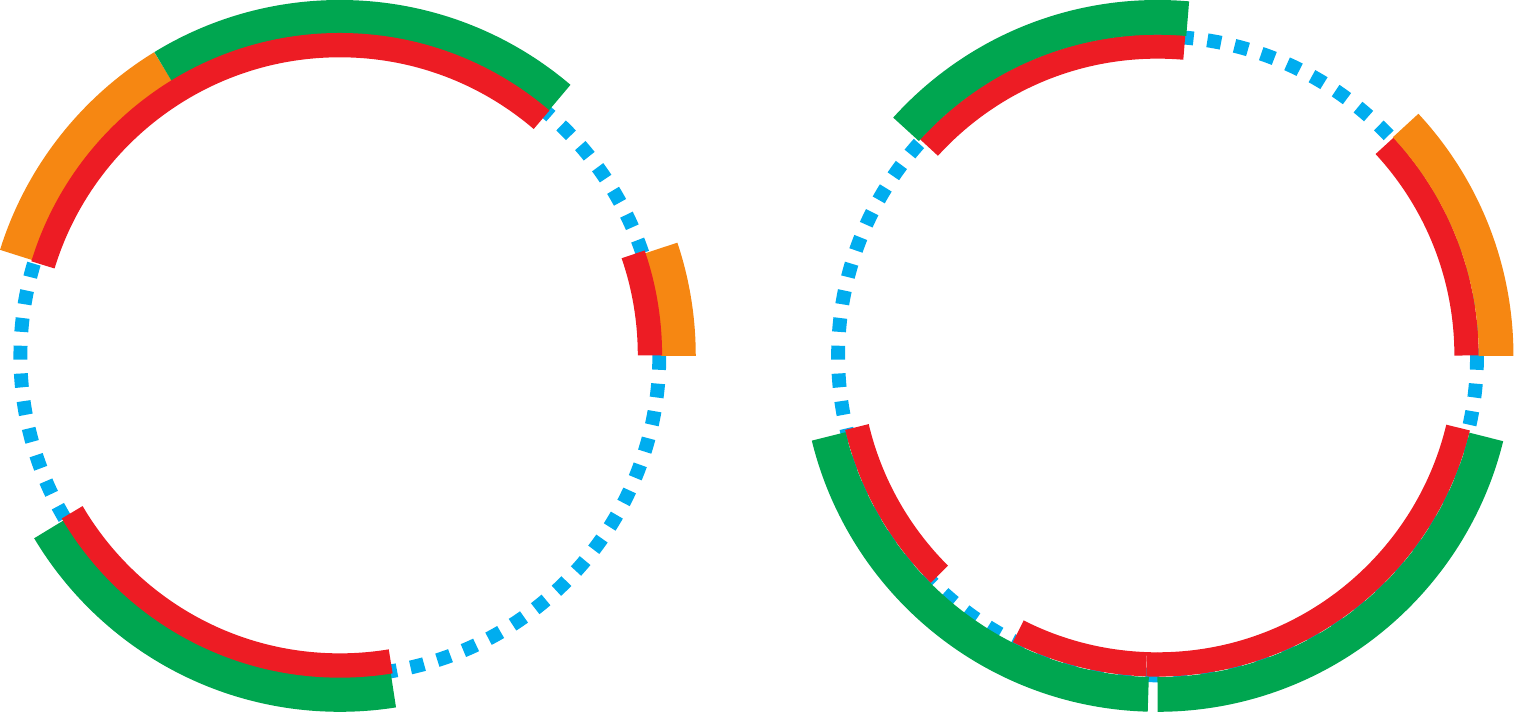}
    \caption{An \opglr problem and an associated optimal solution. The 
		problem has two perimeters and $t = 2$ with $n_1$=3, $n_2$=5, 
		$a_1$=5, $a_2$=8. The boundaries are shown as circles for ease of 
		illustration.
		}
		\label{fig:opglrm}
\end{figure}

Shifting our attention to \opgmc, Fig.~\ref{fig:opgmc} illustrates the 
structure of an optimal solution to a problem with three types of robots 
with capacities and costs being $\ell_1=11, c_1=2$, $\ell_t=30, 
c_2=4$, and $\ell_3=55, c_3=7$, respectively. In this case, the majority 
of the deployed robots are of type $2$ with $\ell_2=30, c_2=4$. Only one 
type $1$ and one type $3$ robots are used. The four perimeter segments are 
covered by three robot groups. 
%
%
The only type $3$ robot guards (the purple segment) across two different 
perimeter segments. Coverage length adjustment is also performed to avoid 
the unnecessary coverage of some gaps. 

\begin{figure}[!ht]
    \centering
    \includegraphics[scale = 0.4]{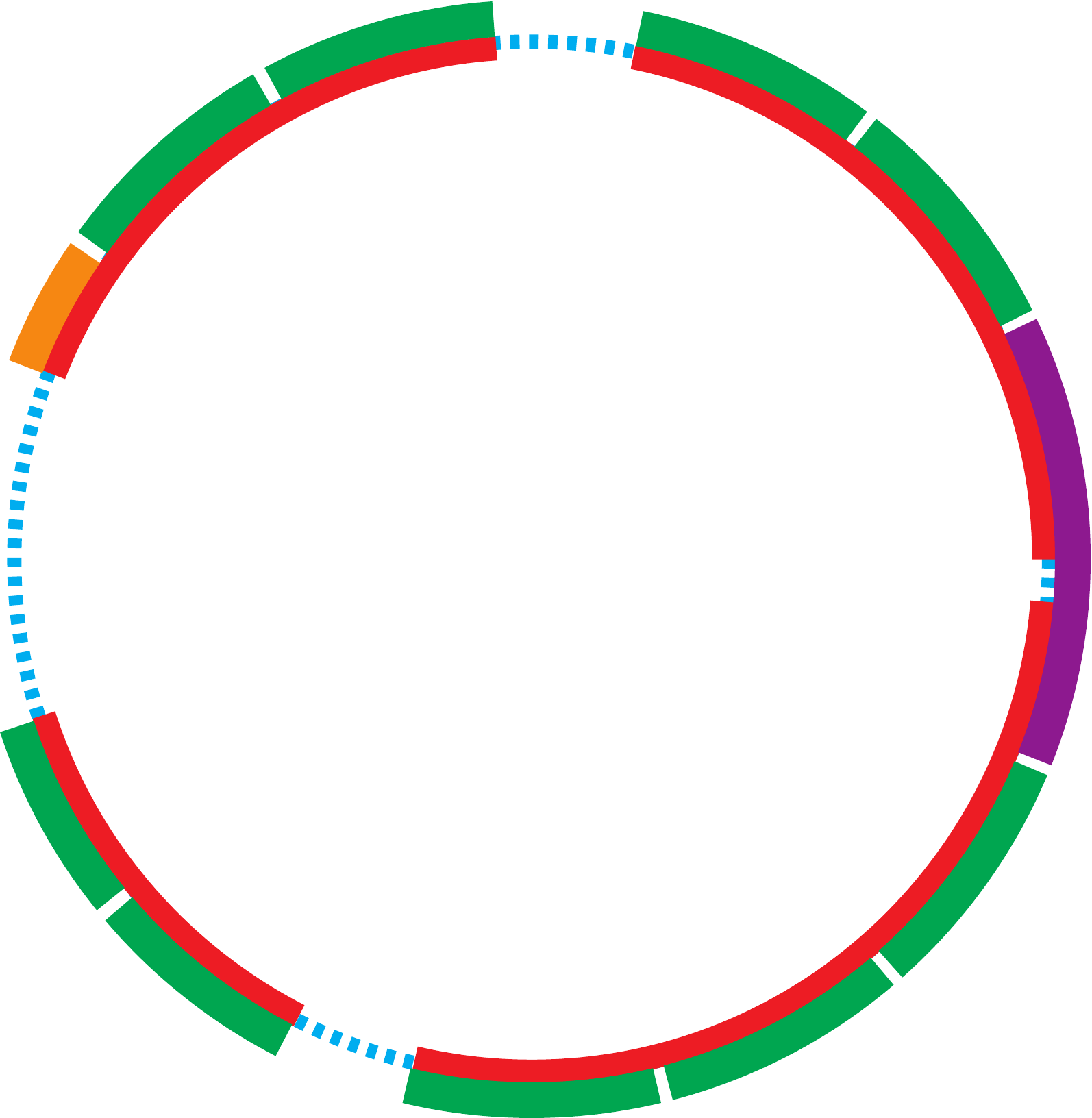}
    \caption{An \opgmc problem and an associated optimal solution. The 
		problem has four (red) perimeter segments and three types of robots
		with $\ell_1=11, c_1=2$ (orange), $\ell_t=30, c_2=4$ (green), 
		and $\ell_3=55, c_3=7$ (purple), respectively.}
		\label{fig:opgmc}
\end{figure}

\subsection{A Robotic Guarding and Patrolling Application}
In this subsection, as a potential application, the DP algorithms for 
\opglr and \opgmc are employed to solve the problem of securing the 
perimeter of the Edinburgh castle, an example used in 
\cite{FenHanGaoYu19RSS}. As shown in Fig.~\ref{fig:castle} (minus 
the orange and green segments showing the solutions), the central 
region of the Edinburgh castle has tall buildings on its boundary 
(the blocks in brick red); these parts of the boundary are the gaps 
that do not need guarding. In the figure, the top sub-figure shows 
the optimal solution for an \opglr instance and an \opgmc instance with 
a total of $11$ robots. The bottom sub-figure is a slightly updated 
\opgmc instance with slightly higher $c_2$. 
\begin{figure}[!ht]
    \centering
    \includegraphics[scale = 0.5]{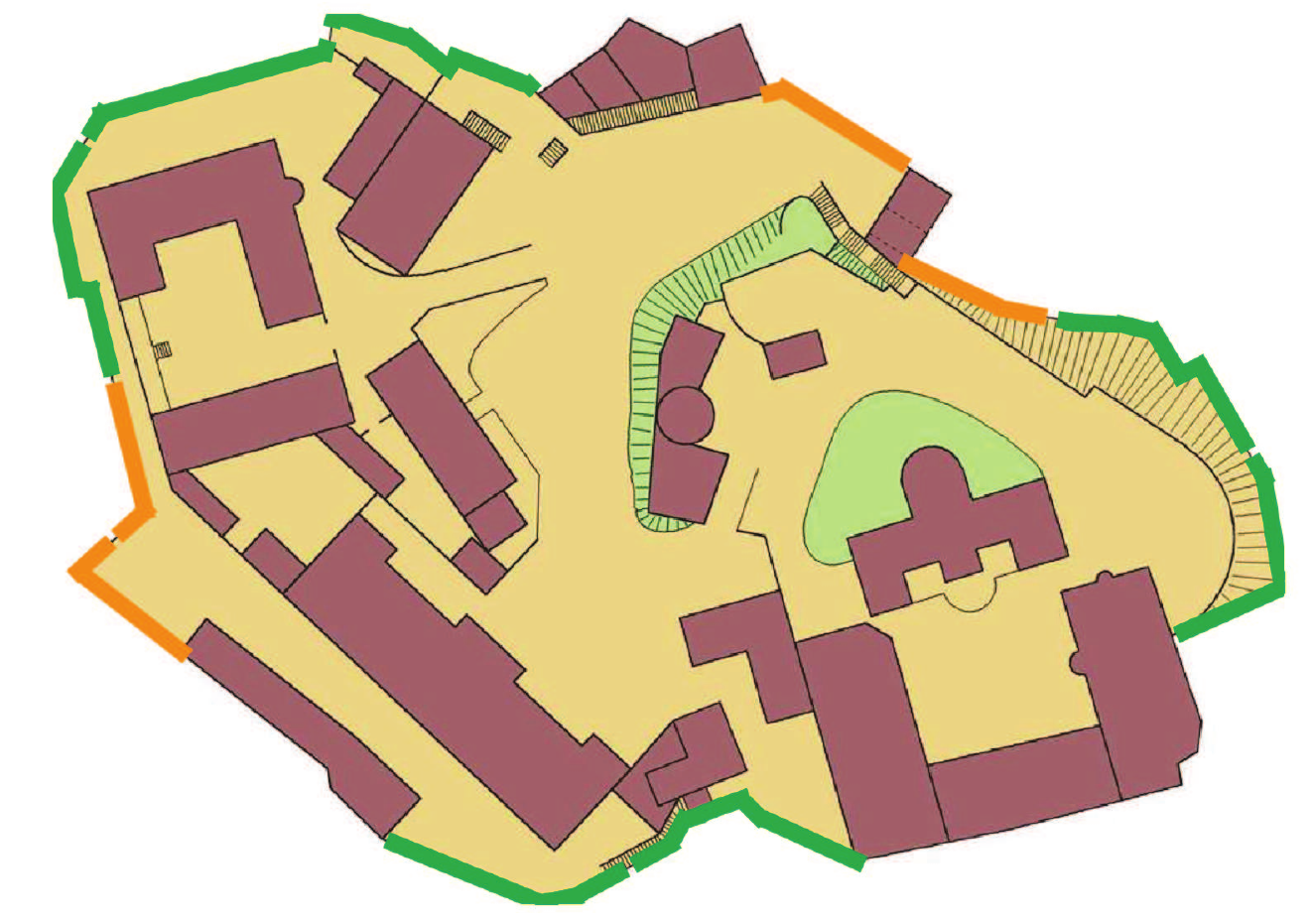}
    \includegraphics[scale = 0.5]{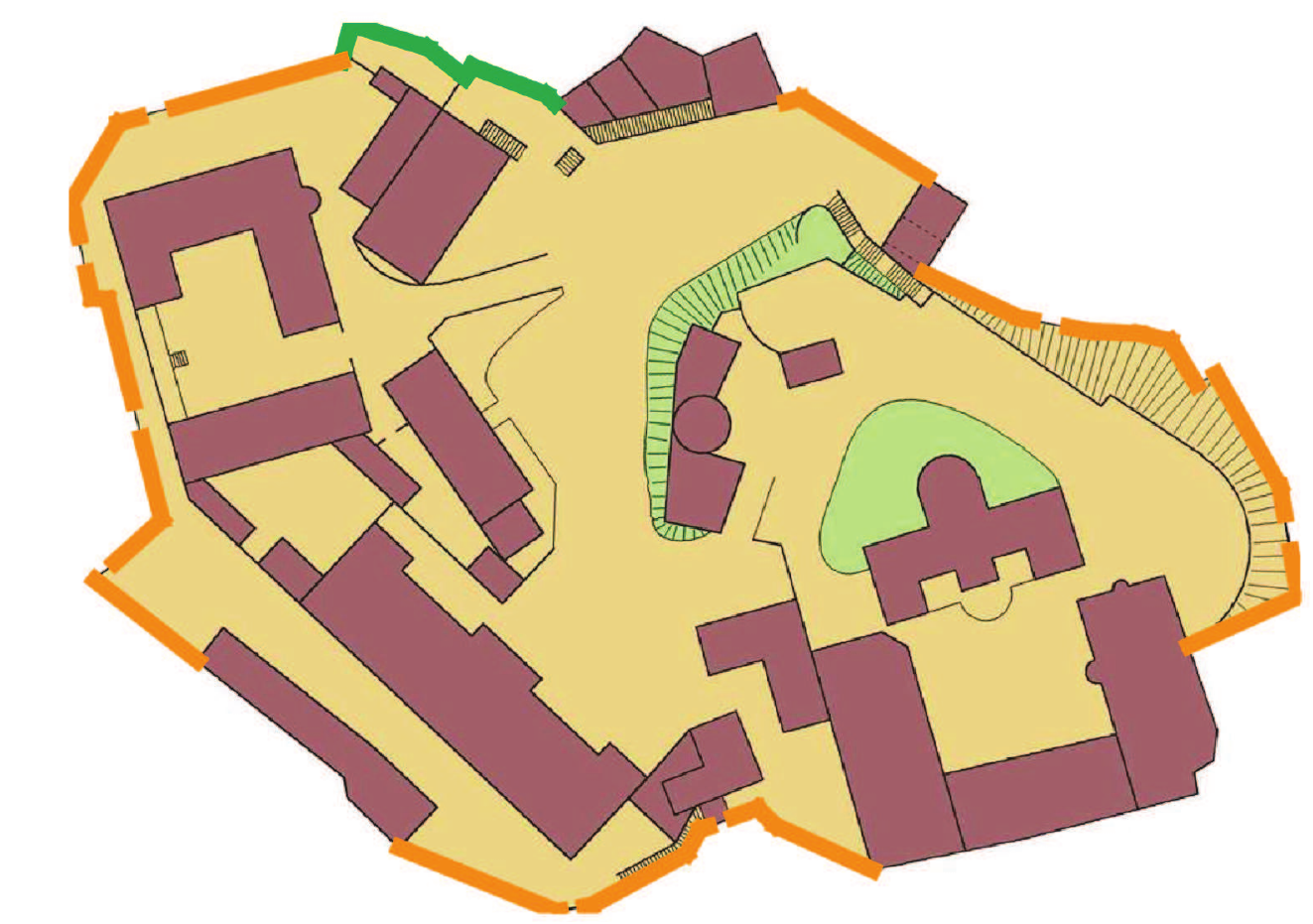}
    \caption{[left] \opglr solution with $n_1 = 4, n_2 = 7, 
		c_1:c_2 = 2:3$ and \opgmc solution with $\ell_1 = 150, 
		c_1 = 100, \ell_2 = 225, c_2 = 145$, and total boundary $3058$. Cost 
		of \opgmc solution is $1415$.
		[right] \opgmc solution with $\ell_1 = 150, c_1 = 100, \ell_2 = 
		225, c_2 = 155$. Cost of solution ($13$ type $1$, $1$ type $2$) is 
		$1455$.
		In both solutions, covers by type $1$ (resp., 
		type $2$) robots are shown in orange (resp., green).
		}
		\label{fig:castle}
\end{figure}

It can be observed that the results, while having non-trivial structures, 
make intuitive sense. For the top sub-figure, solutions to both \opglr 
and \opgmc (because robot with larger capacity is slightly lower in 
relative cost) use mainly higher capacity robots to cover longer perimeter 
segments and use the lower capacity robots mostly fillers. The solution 
covers a small gap at the bottom. For the bottom sub-figure, while only 
small changes are made to the cost, because the longer segment is more 
expensive to use now, the first type of robot is used mainly. 

\subsection{Computational Performance}
With Section~\ref{sec:algorithm} fully establishing the correctness and 
asymptotic complexity of the pseudo-polynomial time algorithms, here, the 
running time of these algorithms are experimentally evaluated. In doing 
so, the main goal is demonstrating that, despite the hardness of \opglr 
and \opgmc, the proposed algorithms could solve the target problems under 
reasonably broad settings in a scalable way. For results presented in 
this subsection, each data point is an average over 10 randomly generated 
instances. 

The first two numerical evaluations (Table~\ref{tab:opglr} and 
Table~\ref{tab:mopglr}) focus on the running times of the pseudo-polynomial 
time algorithms for \opglr over single and multiple perimeters, 
respectively. In these two tables, $t$ and $q$ are the number of types 
and the number of segments, respectively. For each type $\tau$, a 
capacity ($a_{\tau}$) is randomly sampled as an integer between $1$ and 
$100$, inclusive. The number of robots available for each type ($n_{\tau}$) 
is sampled uniformly between $5$ and $15$, inclusive. For the multiple 
perimeters case, the parameter $m$ represents the number of perimeters for 
a given instance.

For the single perimeter case (Table~\ref{tab:opglr}), the results show 
that the pseudo-polynomial time algorithm is effective for up to five 
types of robots, for dozens of robots. We expect a more efficient 
(e.g., C++ based) implementation should be able to effectively handle 
up to five types of robots with the total number of robots being around 
a hundred, on a typical PC. This is likely sufficient for many practical 
applications which have limited types and numbers of robots. Since the 
algorithm has exponential dependency on $t$, it becomes less efficient 
for larger $t$ as expected.  

\begin{table}[htbp]
	\centering
	\begin{tabularx}{\columnwidth}{|c|X|X|X|X|X|X|}
		\hline
		\diagbox{$t$}{$q$}&  \quad 5 &   \quad 10 &\quad 20& \quad 30 & \quad 40&\quad 50 \\
		\hline
		\renewcommand{\arraystretch}{1.05}
		2&0.022 &0.044 &0.131 &0.208 &0.326 &0.516 \\\hline
        3&0.281 &0.714 &1.670 &2.577 &4.107 &4.708 \\\hline
        4&5.504 &16.07 &41.68 &71.55 &109.9 &138.9 \\\hline
        5&29.53 &75.60 &243.6 &443.4 &528.0 &725.0 \\\hline
	\end{tabularx}
	\caption{Running time in seconds used by the DP algorithm for \opglr over 
	a single perimeter.
	}
	\label{tab:opglr}
\vspace*{-1mm}
\end{table}

Table~\ref{tab:mopglr} illustrates the running time of the DP algorithm for 
\opglr over multiple perimeters. As can be readily observed, the impact of 
the number of perimeters $m$ on the running time is relatively small; the 
number of robot types is still the determining factor for running time. In 
this case, our proposed solution is effective for $t$ up to $4$ and starts to 
slow down a robot types become larger than $4$. 
\begin{table}[htbp]
	\centering
	\renewcommand{\arraystretch}{1.05}
    \begin{tabularx}{\columnwidth}{|c|X|X|X|X|X|X|}
        \hline
        {\multirow{2}{*}{\diagbox{$m$}{$q$}} }&\multicolumn{2}{c|}{10}&\multicolumn{2}{c|}{20}&\multicolumn{2}{c|}{30} \\
        \cline{2-7}
         &\,\,\,$t$=3 & $\,\,\,t$=4& $\,\,\,t$=3 & $\,\,\,t$=4& \,\,\,$t$=3  & \,\,\,$t$=4\\
        \hline
        2&3.148 &133.2 &7.077 &198.4 &10.33 &260.0 \\\hline
        3&4.828 &194.1 &10.125 &290.6 &15.52 &376.7 \\\hline
        4&6.131 &256.8 &12.485 &381.3 &19.75 &514.3 \\\hline
        5&7.622 &321.7 &15.355 &476.2 &24.31 &605.8 \\\hline
    \end{tabularx}
    \caption{Running time in seconds used by the DP algorithm for \opglr over multiple perimeters.}
    \label{tab:mopglr}
\vspace*{-1mm}
\end{table}

Table~\ref{tab:opgmc} provides performance evaluation of \opgmcdp. Since 
there is no difference between single and multiple perimeters for \opgmc,
only problems with single perimeters are attempted. Here, for each robot 
type, the cost is an integer randomly sampled between $1$ and $20$, and 
the capacity is computed as five times the cost plus a random integer 
between $1$ and $20$. In the table, $L = \partial R$, the total length 
of the entire boundary. 
Given \opgmc's lower computational complexity, the DP algorithm, 
\opgmcdp, can effectively deal with over a few hundred types of robots 
with ease. 
\begin{table}[!ht]
	\centering
	\renewcommand{\arraystretch}{1.04}
    \begin{tabularx}{\columnwidth}{|c|X|X|X|X|X|X|}
        \hline
        \multirow{2}{*}{\diagbox{$t$}{$L$}} & 
        \multicolumn{2}{c|}{$10^2$}&\multicolumn{2}{c|}{$10^4$} &\multicolumn{2}{c|}{$10^6$} \\
        \cline{2-7}
        &$q$=20&$q$=50&$q$=20&$q$=50&$q$=20&$q$=50\\
        \hline
        3&0.006 &0.064 &0.041 &0.098 &3.040 &3.144 \\\hline
        10&0.005 &0.066 &0.094 &0.155 &9.423 &9.409 \\\hline
        30&0.009 &0.070 &0.261 &0.320 &26.10 &28.59 \\\hline
        100&0.014 &0.077 &0.910 &0.969 &91.28 &93.20 \\\hline
        300&0.030 &0.091 &2.652 &2.938 &275.6 &270.7 \\\hline
    \end{tabularx}
    \caption{Running time in seconds used by \opgmcdp algorithm.}
    \label{tab:opgmc}
\end{table}
\vspace*{-1mm}

\section{Conclusion and Discussions}\label{sec:conclusion}
In this paper, we investigate two natural models of optimal perimeter 
guarding using heterogeneous robots, where one model (\opglr) limits 
the number of available robots and the second (\opgmc) seeks to 
optimize the total cost of coverage. 

These formulations have many potential applications. One application 
scenario we envision is the deployment of multiple agents or robots 
as ``emergency responders'' that are constrained to travel on the 
boundary. An optimal coverage solution will then translate to minimizing 
the maximum response time anywhere on the perimeter (the part that 
needs guarding). The scenario applies to \opg, \opglr, and \opgmc. 

Another application scenario is the monitoring of the perimeter 
using robots with different sensing capabilities. A simple heterogeneous 
sensing model here would be robots equipped with cameras with different 
resolutions, which may also be approximated as discs of different radii. 
The model makes sense provided that the region to be covered is much 
larger than the sensing range of individual robots and assuming that the 
boundary has relatively small curvature as compared to the inverse of the 
radius of the smallest sensing disc of the robots. For boundary with 
relatively small curvature, our solutions would apply well to the sensing 
model by using the diameter of the sensing disc as the 1D sensing range. 
As the region to be covered is large, covering the boundary will require
much fewer sensors than covering the interior.

On the computational complexity 
side, we prove that both \opglr and \opgmc are NP-hard, with \opglr 
directly shown to be strongly NP-hard. This is in stark contrast to 
the homogeneous case, which admits highly efficient low polynomial 
time solutions \cite{FenHanGaoYu19RSS}. The complexity study also 
establishes structural similarities between these problems and 
classical NP-hard problems including \tpart, \ttkp, and \subsetsum.

On the algorithmic side, we provide methods for solving both \opglr 
and \opgmc exactly. For \opglr, the algorithm runs in pseudo-polynomial 
time in practical settings with limited types of robots. In 
this case, the approach is shown to be computationally effective. 
For \opgmc, a pseudo-polynomial time algorithm is derived for the 
general problem, which implies that \opgmc is weakly NP-hard. In 
practice, this allows us to solve large instances of \opgmc. We 
further show that a polynomial time algorithm is possible for 
\opgmc when the types of robots are fixed. 

With the study of \opg \cite{FenHanGaoYu19RSS} for homogeneous and 
heterogeneous cases, some preliminary understanding has been 
obtained on how to approach complex 1D guarding problems. 
Nevertheless, the study so far is limited to {\em one-shot} settings
where the perimeters do not change. In future research, we would like 
to explore the more challenging case where the perimeters evolve 
over time, which requires the solution to be dynamic as well. Given 
the results on the one-shot settings, we expect the dynamic setting
to be generally intractable if global optimal solutions are desired, 
potentially calling for iterative and/or approximate solutions. 

We recognize that our work 
does not readily apply to a visibility-based sensing model, which is also 
of interest. Currently, we are also exploring covering of the interior
using range-based sensing. As with the OPG work, we want to push for 
optimal or near-optimal solutions when possible.
\vspace*{-1mm}

\bibliographystyle{IEEEtran}
\bibliography{../bib/bib}

\end{document}